\documentclass[11pt]{article}
\usepackage{latexsym}  
\usepackage{amsmath,amssymb,amsfonts}
\usepackage{enumerate}

\begin{document}

\title{Ehresmann-Schein-Nambooripad theorems for classes of biunary semigroups}
\author{Tim Stokes}

\date{}
\maketitle

\newcommand{\bea}{\begin{eqnarray*}}
\newcommand{\eea}{\end{eqnarray*}}

\newcommand{\ben}{\begin{enumerate}}
\newcommand{\een}{\end{enumerate}}

\newcommand{\bi}{\begin{itemize}}
\newcommand{\ei}{\end{itemize}}

\newenvironment{proof}{\noindent \textbf{Proof.}\hspace{.7em}}
                   {\hfill $\Box$
                    \vspace{10pt}}

\newcommand{\mc}{\mathcal}

\newtheorem{thm}{Theorem}[section]
\newtheorem{theorem}[thm]{Theorem}
\newtheorem{lem}[thm]{Lemma}
\newtheorem{pro}[thm]{Proposition}
\newtheorem{proposition}[thm]{Proposition}
\newtheorem{cor}[thm]{Corollary}
\newtheorem{corollary}[thm]{Corollary}
\newtheorem{eg}[thm]{Example}
\newtheorem{dfn}[thm]{Definition}
\newtheorem{prodef}[thm]{Proposition/Definition}

\newcommand{\dom}{\mbox{dom}}
\newcommand{\ran}{\mbox{ran}}

\begin{abstract}
We obtain an ESN theorem for a very general class of biunary semigroups with idempotent-valued domain and range operations, representing them in terms of small categories equipped with a suitable biaction of the identities on the category.  Our results generalise the recent work of Fitzgerald and Kinyon connecting localisable semigroups to transcription categories, as well as that of Lawson linking Ehresmann semigroups to categories with Ehresmann biaction.  In contrast to most approaches to ESN theorems, we do not require the categories to be ordered or for their sets of identities to possess any particular structure.  Throughout, the biunary semigroups are represented using categories rather than generalised categories of any kind, and we obtain category isomorphisms between the clesses of semigroups and their associated enriched categories, rather than category equivalences.  Our results cover the class of DRC-semigroups considered by Jones and Shoufeng Wang, but they also cover cases where not both congruence conditions hold, including examples such as the semigroup of binary relations on a set under demonic composition equipped with domain and range operations.
\end{abstract}

\medskip


\section{Introduction}  \label{sec:intro}

Throughout, if $X$ is a non-empty set then $Rel(X)$ denotes the semigroup of binary relations on $X\times X$ under familiar relational composition, $PT(X)$ denotes the semigroup of partial functions $X\rightarrow X$, and $I(X)$ the inverse semigroup of one-to-one partial functions on $X$.  Function application is always written on the right (with the exception of unary operation application), and composition is to be read left-to-right.

For current purposes, a small category $C$ is a set equipped with a partial binary operation $\circ$ and two unary operations denoted by $D$ and $R$, satisfying, for all $x,y\in C$,
\ben[ (C1)]
\item $D(x)\circ x=x$, $x\circ R(x)=x$
\item $R(D(x))=D(x)$, $D(R(x))=R(x)$
\item $x\circ y$ exists if and only if $R(x)=D(y)$
\item if $R(x)=D(y)$ then $D(x\circ y)=D(x)$ and $R(x\circ y)=R(y)$
\item $x\circ (y\circ z)=(x\circ y)\circ z$ whenever the two products are defined.
\een
This ``object-free" formulation (or some equivalent of it) is frequently used in algebra; it is the definition used in \cite{fitzkin}, and is equivalent to that used in \cite{law1}.  If $C$ is a category in the above sense, denote by $C^0$ the identities of $C$ -- elements of the form $D(s)$ (or equivalently $R(s)$ by the second law), where $s\in C$.

In terms of biunary semigroups, the starting point for the current work is the class of {\em localisable semigroups}, defined in \cite{fitzkin} by the following laws.
\bi
\item $D(x)x=x$
\item $D(xy)=D(xD(y))$
\item $D(x)D(y)=D(D(x)y)$
\item $xR(x)=x$
\item $R(xy)=R(R(x)y)$
\item $R(x)R(y)=R(xR(y))$
\item $R(D(x))=D(x), D(R(x))=R(x)$
\ei
It follows easily that for a localisable semigroup $S$, $D(D(x))=D(x)$ and $R(R(x))=R(x)$ for all $x\in S$, and that $D(S)=\{D(s)\mid s\in S\}$ is a band.  As noted in Corollary 5.3 of \cite{fitzkin}, Ehresmann semigroups are nothing but localisable semigroups in which the projections commute: they satisfy the further law $D(x)D(y)=D(y)D(x)$, and so $D(S)$ is a semilattice (hence a poset if we define $e\leq f$ if and only if $e=ef$, in the usual way).  It is well-known that the semigroup $Rel(X)$ of binary relations on the non-empty set $X$, equipped with domain and range operations given by $D(\rho)=\{(x,x)\mid x\in \dom(\rho)\}$ and dually for $R(\rho)$ in terms of $\ran(\rho)$, is an Ehresmann semigroup.  Natural examples of non-Ehresmann localisable semigroups are harder to come by, as noted in \cite{fitzkin}.

In any localisable semigroup $S$, one may define the partial binary operation $\circ$ as follows: for all $s,t\in S$,
$$s\circ t=st\mbox{ providing $R(s)=D(t)$,}$$
and undefined otherwise.  Then the structure $C={\mc C}(S)=(S,\circ,D,R)$ is easily seen to be a category in which $C^0=D(S)$.   

Before the advent of localisable semigroups, Lawson considered this category construction, and was able to characterise those categories ${\mc C}(S)$ arising from an Ehresmann semigroup $S$ in this way.  He did so by enriching the structure of $C={\mc C}(S)$ by defining two partial orders $\leq_l,\leq_r$ on $C$ as follows: $s\leq_l t$ if and only if $s=tR(s)$, and $s\leq_r t$ if and only if $s=D(s)t$; note that on $C^0=D(S)$, these two partial orders agree with the usual semilattice order on $D(S)$.  Then on $C$, notions of ``restriction" and ``co-restriction" can be defined: for $e\in C^0$ and $s\in C$.  One defines the restriction $e|s=es$ whenever $e\leq D(s)$, and dually for the co-restriction $s|e$.  Then $e|s$ can be shown to be the unique $t\in C$ such that $t\leq_r s$ and $D(t)=e$, and dually for $s|e$.  Ehresmann then characterised Ehresmann semigroups as small categories equipped with two partial orders in which notions of restriction and co-restriction exist, satisfying various laws; for the details, consult \cite{law1}.  This is a generalisation of the original ESN theorem linking inverse semigroups to inductive$_1$ groupoids.

In other related work, sometimes a variation on the partial operation $\circ$ given above is used, which is defined ``more often" and gives rise to a generalised category structure.  This happens in the work of Jones in \cite{jones} and Wang in \cite{Swang}.  These approaches make use of generalised categories which come equipped with two partial orders and notions of restriction and co-restriction analogous to those used in \cite{law1}, and whose identities form some kind of ``projection algebra" rather than a semilattice.  In these cases, the category of biunary semigroups under consideration is shown to be isomorphic to a corresponding category of enriched generalised categories. 

There is important related work involving non-biunary semigroups.  In one of the most important contributions to the theory of regular semigroups, Nambooripad connected regular semigroups to inductive$_2$ groupoids (see \cite{Nambooripad}).  Following this, Armstrong in \cite{Armstrong} connected concordant semigroups to inductive$_2$ cancellative categories.  In these two cases, the corresponding categories of semigroups and of categories are equivalent but not isomorphic: in particular, the underlying set of the category is generally different to that of the semigroup to which it is equivalent.  In other work, Gould and Wang \cite{gw} obtained a category isomorphism between the class of weakly B-orthodox semigroups and suitable generalised categories, and then Wang in \cite{wang1} obtained an equivalence between the same class of semigroups and a class of actual categories.   

In all of the above approaches, the (possibly generalised) categories corresponding to the class of semigroups of interest are equipped with partial orders  in terms of which notions of restriction and co-restriction are defined; moreover, the identities in these (possibly generalised) categories are assumed to have some algebraic structure consistent with the order(s) on the (generalised) category (for example, that of a semilattice in the original ESN theorem and in \cite{law1}, or of projection algebra in \cite{Swang}, or indeed of regular biordered set in Nambooripad's work).  

In \cite{fitzkin}, the authors obtained an ESN theorem that related the class of localisable semigroups to a class of enriched categories.  In general, a localisable semigroup $S$ has no natural partial orders definable on it in terms of which the structure of $C={\mc C}(S)$ might be characterised.  However, we may define $e|s=es$ and $s|e=se$ for {\em all} $e\in C^0$ and $s\in C$ (rather than only when $e\leq D(s)$ or $e\leq R(s)$).  Motivated by this, the authors of \cite{fitzkin} were led to define a {\em transcription category} to be a category $C$ equipped with left and right actions of the identities of the category $C^0$ on all of $C$, here denoted $e|s,s|e$ for all $e\in C^0$ and $s\in C$, and satisfying the following.

\begin{enumerate}[ (TC1)] 
\item For $e,f\in C^0$, $e|f$ does not depend on which way the action is interpreted. 
\item For all $a\in C$, $D(a)|a = a$ and $a|R(a) = a$.
\item For all $a\in C$ and $e,f\in C^0$, $e|(f|a) = (e|f)|a$ and $a|(e|f)=(a|e)|f$.
\item For all $a,b\in C$, if $a\circ b$ exists then for all $e\in C^0$, 
\ben [ (TC4a)] 
\item so does $(e|a)\circ R(e|a)|b$, and $e|(a\circ b)=(e|a)\circ R(e|a)|b$;
\item so does $a|D(b|e)\circ b|e$, and $(a\circ b)|e=a|D(b|e)\circ b|e$.
\een
\item For all $e\in C^0$ and $a\in C$, 
\ben [ (TC5a)]
\item $D(e|a)=e|D(a)$; 
\item $R(a|e)=R(a)|e$.
\een
\item For all $e,f\in C^0$ and $a\in C$, $(e|a)|f = e|(a|f)$.
\een

The properties above were labelled (3.1a)--(3.1f) at the beginning of the third section of \cite{fitzkin}.  It follows from (TC5a) or (TC5b) that $e|f\in C^0$ for all $e,f\in C^0$; hence the two laws in (TC3) make sense.  Likewise, it follows from the laws other than (TC4) that if $a\circ b$ exists, then so do $(e|a)\circ R(e|a)|b$ and $a|D(b|e)\circ b|e$, and so (TC4a) does not strictly require the assumption that $(e|a)\circ R(e|a)|b$ exist, and dually for (TC4b); however, in the more general settings considered in what follows, the form stated above is required.  

It was shown in \cite{fitzkin} that if $S$ is a localisable semigroup, then $C={\mc C}(S)$ is a transcription category when equipped with the biaction described above ($e|s=es$ and $s|e=se$ for all $s\in C$ and $e\in C^0$).  Conversely it was shown that, given a transcription category $C$, one can turn it into a localisable semigroup by retaining $D$ and $R$ but defining a pseudoproduct via $s\otimes t=s|D(t)\circ R(s)|t$ for all $s,t\in C$ (noting that this pseudoproduct always exists).  These constructions are shown to be mutually inverse in \cite{fitzkin} and indeed one can obtain an isomorphism of categories with morphisms defined in the natural ways, as follows easily from Theorem $4.8$ in \cite{fitzkin}.
  
Lawson's definition of Ehresmann biactions on categories given in \cite{law2} uses slightly different but equivalent defining laws in place of (TC1)--(TC6), with the law $e|f=f|e$ for all $e,f\in D(S)$ added.  Lawson in \cite{law2} showed that Ehresmann semigroups correspond to categories with an Ehresmann biaction, a result which can be viewed as a special case of the main result of \cite{fitzkin} that links localisable semigroups to transcription categories.   

It is this approach that we shall generalise in what follows.  For a biunary semigroup with sufficient structure, we show how there is an associated category equipped with a biaction defined in a way formally identical to that used to define transcription categories from localisable semigroups.  We are able to characterise the resulting categories with biaction, and obtain category isomorphisms between the biunary semigroups and the categories with biaction.  We consider several special cases of the correspondence.

Throughout, we work with actual categories rather than generalisations of them, and obtain isomorphisms between the categories of biunary semigroups and of enriched categories in all cases, rather than equivalences.  Our results apply to the DRC-semigroups of \cite{jones}, obtaining an alternative ESN theorem to that given in \cite{Swang}.  But they also apply to cases in which not both congruence conditions hold, including a class containing the biunary semigroup of binary relations on a non-empty set under so-called demonic composition to which we return later.

\section{Precat-semigroups and their various types}

\subsection{Defining precat-semigroups}

We next define the most general types of semigroups for which an approach formally similar to that taken in \cite{fitzkin} and \cite{law2} can be used.

\begin{dfn}
A biunary semigroup $S$, with unary operations $D$ and $R$, is a {\em precat-semigroup} if for all $x\in S$,
\ben[ (CS1)]
\item $D(x)x=x$
\item $xR(x)=x$
\item $D(x)^2=D(x)$
\item $D(R(x))=R(x)$
\item $R(D(x))=D(x)$
\een
Elements of $D(S)=\{D(s)\mid s\in S\}$ are called {\em projections}.
\end{dfn}

In particular, localisable semigroups are precat-semigroups, but there are many more examples than these.  As for localisable semigroups, it follows easily that $D(D(x))=D(x)$ and $R(R(x))=R(x)$.  A key difference is that $D(S)$ need not be a band.

If $S$ is a precat-semigroup, note that $D(S)\subseteq E(S)$ (the set of idempotents of $S$), and for all $e\in D(S)$, $D(e)=R(e)=e$; hence $R(S)$, defined dually to $D(S)$, is equal to it.  Put simply then, a precat-semigroup is a biunary semigroup with a distinguished set of idempotents $D(S)$ consisting of the elements fixed by the unary operations $D$ and $R$, and which are such that for each $s\in S$, $D(s)$ is a left identity for $s$ and $R(s)$ is a right identity for $s$.

We next list several other important possible properties of precat-semigroups that will feature in the work to follow.

\begin{dfn}
Let $S$ be a precat-semigroup.  Then we say that  
\bi 
\item $S$ is a {\em cat-semigroup} if it satisfies  
\ben[ (CS6)]  
\item for all $x,y$, $R(x)=D(y) \Rightarrow (D(xy)=D(x)\ \&\ R(xy)=R(y))$;
\een
\item $S$ satisfies the {\em left (respectively right) congruence condition} if it satisfies the law 
$$D(xy)=D(xD(y))\mbox{ (respectively }R(xy)=R(R(x)y)),$$ and it is said to satisfy the {\em congruence conditions} if it satisfies both the left and right congruence conditions;
\item $S$ satisfies the {\em left (respectively right) weak congruence condition} if it satisfies the law $$D(xy)=D(xD(R(x)y)))\mbox{ (respectively }R(xy)=R(R(xD(y))y)),$$ and it satisfies the {\em weak congruence conditions} if it satisfies both;
\item $S$ satisfies the {\em strong match-up conditions} if for all $x,y\in S$, $R(xD(y))=D(R(x)y)$ and $xy=xD(y)R(x)y$;
\item $S$ satisfies the {\em match-up conditions} if it satisfies both 
\bi
\item  the law $R(xD(y))=D(R(xD(y))y)$, the {\em left match-up condition}, and
\item the law $D(R(x)y)=R(xD(R(y)x))$, the {\em right match-up condition};
\ei
\item $S$ is {\em left semi-localisable} if it satisfies the left congruence and right weak congruence conditions and $D(S)$ is a band, and we define {\em right semi-localisability} dually;
\item $S$ is {\em D-ample} if it satisfies $xD(y)=D(xy)x$ for all $x,y\in S$;
\item $S$ is a {\em left restriction semigroup with range} if $(S,\cdot,D)$ is a left restriction semigroup and $R$ satisfies $R(xy)R(y)=R(xy)$.
\ei
\end{dfn}

In the remainder of this section, we motivate each of these properties in turn, give  examples, and explore how the properties relate to one-another.  These properties will all feature in our generalised ESN theorems that follow.

\subsection{Cat-semigroups}

Following the approach of \cite{law2} and \cite{fitzkin}, in any precat-semigroup, a partial operation may be defined as follows.

\begin{dfn}  \label{restr}
Let $S$ be a precat-semigroup.  For all $s,t\in S$, define the partial binary operation $\circ$ by setting
$$s\circ t=st\mbox{ providing $R(s)=D(t)$,}$$
and undefined otherwise; this is the {\em restricted product}.  Define ${\mc C}(S)=(S,\circ,D,R)$.
\end{dfn}

Generally, if $S$ is a precat-semigroup, then it is easy to see that ${\mc C}(S)$ is a category if and only if $S$ is a cat-semigroup.  In particular, localisable semigroups are cat-semigroups.

\begin{dfn}
If $S$ is a cat-semigroup then we call ${\mc C}(S)$ the {\em category determined by} $S$.  
\end{dfn}

Unblike the class of precat-semigroups, the class of cat-semigroups is not a variety of biunary semigroups. 

\begin{eg}  \label{catqv}
A cat-semigroup having a quotient that is not a cat-semigroup.
\end{eg} 
\vspace{-5pt}
Let $S=\{a,g,e,1\}\subseteq I(X)$ where $X=\{w,x,y,z\}$ and 
$$a=\{(w,x),(x,w)\},\ g=\{(w.w),(x,x)\},\ e=\{(w.w),(x,x),(y,y)\},$$
and $1$ is the identity function on $X$.  Then $S$ is a subsemigroup of $I(X)$, with multiplication table as follows:
$$	\begin{array}{c|cccc}
	\cdot&a&g&e&1\\
	\hline
	a&g&a&a&a\\
	g&a&g&g&g\\
	e&a&g&e&e\\
	1&a&g&e&1
	\end{array}.
	$$
Clearly, $S$ is commutative, and indeed is an inverse subsemigroup of $I(X)$ (since $a^\prime=a$ with all other elements idempotent), and therefore comes equipped with in-built notions of domain and range.  However, it can also be viewed as a cat-semigroup in which $D(a)=e$, $R(a)=1$, and $D(s)=R(s)=s$ for all other $s\in S$.  When checking this, only (CS6) is not obvious.   But for this, if $R(s)=D(t)$ and $s,t\in D(S)$, then $s=t$ and so $D(st)=D(s^2)=D(s)$, and similarly $R(st)=R(t)$; if $R(s)=D(a)$ for some $s\in S$ then $R(s)=e$ so $s=e$ and so $D(sa)=D(ea)=D(a)=e=D(e)$; and if $R(a)=D(s)$ for some $s\in S$ then $D(s)=1$ so $s=1$, and so $D(as)=D(a)$ and $R(as)=R(a)=1=R(s)$.

Now $S$ has a semigroup congruence collapsing $e,1$ and respecting $D$ and $R$, as is easily seen.  The resulting quotient $S'=\{\{a\},\{g\},\{e,1\}\}$ is a precat-semigroup (since these form a variety), but has $D(\{a\})=R(\{a\})=\{e,1\}$, yet $R(\{a\}^2)=R(\{g\})=\{g\}\neq R(\{a\})$, so (CS6) fails and so $S'$ is not a cat-semigroup.  

\begin{cor} 
The class of cat-semigroups is a proper quasivariety.
\end{cor}

\subsection{The congruence conditions}

It follows immediately from what is noted in Section 2.1 of \cite{ordehrs} that localisable semigroups are nothing but precat-semigroups satisfying the congruence conditions in which the projections form a band (a semilattice in the case of Ehresmann semigroups); indeed, any band $S$ can be turned into a localisable semigroup if we define $D(s)=R(s)=s$ for all $s\in S$.  

It also follows that a precat-semigroup $S$ is localisable if and only if it is both left semi-localisable and right semi-localisable.  (Note that the term ``left localisable" was used in \cite{fitzkin} for a unary semigroup satisfying the defining laws for localisable semigroup that involve $D$ but not $R$, and dually for ``right localisable".)

Clearly, the congruence conditions imply the weak congruence conditions within precat-semigroups.  It is immediate that if a precat-semigroup satisfies the congruence conditions then it is a cat-semigroup and so ${\mc C}(S)$ is a category, but in fact the weak congruence conditions are sufficient.

\begin{pro}  \label{fbv}
If a precat-semigroup satisfies the weak congruence conditions then it is a cat-semigroup.  Hence, the class of cat-semigroups satisfying the (weak) congruence conditions is the finitely based variety of precat-semigroups satisfying them.
\end{pro}
\begin{proof}
Suppose $R(x)=D(y)$ in the precat-semigroup $S$ satisfying the weak congruence conditions.  Then $$D(xy)=D(xD(R(x)y)))=D(xD(D(y)y)))=D(xD(y))=D(xR(x))=D(x).$$  
Dually, $R(xy)=R(y)$. 
\end{proof}

\subsection{The match-up conditions}

The various match-up conditions are motivated by the desire to capture the multiplication operation in a cat-semigroup $S$ in terms of the category it determines, $C={\mc C}(S)$.  Following \cite{fitzkin} and \cite{law2}, one may enrich $C={\mc C}(S)$ by retaining arbitrary products of projections with semigroup elements, so we define a ```biaction" of $C^0$ on $C$ via $e|s=es$ and $s|e=se$ for all $e\in C^0=D(S)$ and $s\in C=S$.  Note that we use the same notation for both actions, since there is only ambiguity when both arguments are from $D(S)$, and then the interpretation does not matter!  

Recall that for a localisable semigroup $S$, the semigroup operation can be recovered from ${\mc C}(S)$ equipped with this biaction via $st=(s|D(t))\circ (R(s)|t)$ for all $s,t\in S$.  The same process of recovery of the original cat-semigroup from the category with biaction it determines can take place as long as $(s|D(t))\circ (R(s)|t)$ as just defined exists in ${\mc C}(S)$ and correctly calculates $st$ in $S$; it is easy to see that this requirement is equivalent to the strong match-up conditions.  As we shall see in Subsection \ref{sec:catdet}, there are non-localisable cat-semigroups that satisfy the strong match-up conditions.

We next show that the strong match-up conditions imply the match-up conditions, so our nomenclature is consistent. 

\begin{pro}  \label{match0}
Suppose a cat-semigroup $S$ satisfies the first law in the strong match-up conditions -- $R(sD(t))=D(R(s)t)$. Then $S$ satisfies the match-up conditions.  Hence the strong match-up conditions imply the match-up conditions.
\end{pro}
\begin{proof}
For all $x,y\in S$, $R(xD(y))= R((xD(y))D(y))=D(R(xD(y))y)$, upon using the assumed law with $s=xD(y)$ and $t=y$.  The right match-up condition follows dually.
\end{proof}

Note that in any precat-semigroup, we have $$st=(sD(t))(R(sD(t))t)=(sD(R(s)t))(R(s)t).$$  
Moreover, if (and only if) the the left (respectively right) match-up condition holds, we may express $st$ within ${\mc C}(S)$ as $s|D(t)\circ R(s|D(t))|t$ (respectively $s|D(R(s)|t)\circ R(s)|t$).  

The main significance of the weak congruence conditions is that they can be used to equationally characterise those cat-semigroups satisfying only one of the match-up conditions.

\begin{pro}  \label{MUC2}
The class of cat-semigroups satisfying the left match-up condition is the variety of precat-semigroups satisfying the left congruence and right weak congruence conditions plus the law $R(st)=D(R(st)R(t))$.  
\end{pro}
\begin{proof}
Suppose $S$ is a cat-semigroup satisfying the left match-up condition. Then for all $s,t\in S$, $R(sD(t))=D(R(sD(t))t)$, so letting $x=sD(t)$ and $y=R(sD(t))t$, we see that $xy=sD(t)R(sD(t))t=sD(t)t=st$.  But $R(x)=D(y)$, and so using the first cat-semigroup quasiequation, $D(xy)=D(x)$, so $D(st)=D(sD(t))$; using the second cat-semigroup law, $R(xy)=R(y)$, so $R(st)=R(R(sD(t))t)$.  So $S$ satisfies the left congruence and right weak congruence conditions.  Hence, for all $s,t\in S$,
\bea
R(st) &=& R(stR(t))\\
&=& R(stD(R(t)))\\
&=& D(R(stD(R(t)))R(t))\mbox{ by the left match-up condition}\\
&=& D(R(stR(t))R(t))\\
&=& D(R(st)R(t)).
\eea
Conversely, suppose $S$ is a precat-semigroup satisfying the left congruence and right weak congruence conditions plus the law $R(st)=D(R(st)R(t))$ for all $s,t\in S$.  Then it is a cat-semigroup by Proposition \ref{fbv}.  Moreover, 
\bea
R(sD(t)) 
&=& D(R(sD(t))R(D(t)))\mbox{ by the first additional law}\\
&=& D(R(sD(t))D(t))\\
&=& D(R(sD(t))t)\mbox{ by the left congruence condition.}
\eea
Hence the left match-up condition is satisfied.
\end{proof}

From Proposition \ref{MUC2} and its dual, and the fact that the congruence conditions imply the weak congruence conditions, we obtain the following.

\begin{cor}  \label{matchcong}
The class of cat-semigroups satisfying the match-up conditions is the variety of precat-semigroups satisfying the congruence conditions and the two laws 
\bi
\item $R(st)=D(R(st)R(t))$ and
\item $D(st)=R(D(s)D(st))$.
\ei
\end{cor}

It follows from this result that DRC-semigroups as defined in \cite{jones} and considered from an ESN theorem viewpoint in \cite{Swang} are cat-semigroups satisfying the match-up conditions, since these are precat-semigroups satisfying the congruence conditions and some further laws that obviously imply the laws $R(st)=D(R(st)R(t))$ and $D(st)=R(D(s)D(st))$; specifically, they satisfy the laws $R(st)=R(st)R(t)$ and $D(st)=D(s)D(st)$.  

Like localisable semigroups, DRC-semigroups generalise Ehresmann semigroups.  However, although few natural examples of localisable semigroups that are not Ehresmann are known, DRC-semigroups that are not Ehresmann arise very naturally, for example in connection with operator algebra theory.  As discussed in \cite{DRsemi}, the multiplicative semigroups of Rickart *-rings are DRC-semigroups in which $D(S)$ is reduced.  (Rickart *-rings include all Baer *-rings, which include as examples the rings of all bounded linear operators on a Hilbert space.) 

There are also important examples where only one of the match-up conditions holds but not the other; see Subsection \ref{subsec:Dample}.

\subsection{The strong match-up conditions}  \label{sec:catdet}

\begin{lem}   \label{stronglem}
Let $S$ be a cat-semigroup.  Then $S$ satisfies the law $R(sD(t))=D(R(s)t)$ if and only if it satisfies the congruence conditions and the law $D(ef)=R(ef)$ for all $e,f\in D(S)$.   
\end{lem}
\begin{proof}
Suppose $S$ is a cat-semigroup satisfying the law $R(sD(t))=D(R(s)t)$.  By Proposition \ref{match0}, it satisfies the match-up conditions, and by Corollary \ref{matchcong}, it satisfies the congruence conditions.  Hence, for $e,f\in D(S)$, $R(ef)=R(eD(f))=D(R(e)f)=D(ef)$.  

Conversely, if $S$ satisfies the congruence conditions and $D(ef)=R(ef)$ for all $e,f\in D(S)$, then for all $s,t\in S$, $R(sD(t))=R(R(s)D(t))=D(R(s)D(t))=D(R(s)t)$.
\end{proof}

\begin{cor}  \label{strongmatch}
The class of cat-semigroups $S$ satisfying the strong match-up conditions is the variety of precat-semigroups satisfying the congruence conditions together with the laws $D(ef)=R(ef)$ for all $e,f\in D(S)$ and the law $sD(t)R(s)t=st$.  
\end{cor}

We have already noted that every localisable semigroup satisfies the strong match-up conditions.  The converse fails.

\begin{eg}  \label{strongnotloc} A non-localisable cat-semigroup satisfying the strong match-up conditions.
\end{eg}
Let $S=\{e,f,a\}\subseteq T(X)$, where $X=\{x,y,z\}$ and $e=\{(x,x),(y,z),(z,z)\}$, $f=(x,y),(y,y),(z,y)\}$ and $a=\{(x,z),(y,z),(z,z)\}$.  Then $S$ is a band with multiplication as follows:
$$	\begin{array}{c|ccc}
	\cdot&e&f&a\\
	\hline
	e&e&f&a\\
	f&a&f&a\\

	a&a&f&a
	\end{array}.
	$$
Define unary operations $D$ and $R$ on $S$ by setting $D(a)=e=R(a)$, with $D(s)=R(s)=s$ for $s\neq a$. Then it is routine to check the laws for precat-semigroups are satisfied, and $D(S)=\{e,f\}$.  For the left congruence condition, to check that $D(st)=D(sD(t))$, only cases where $t\not\in D(S)$ are non-trivial, that is, where $t=a$.  But then $D(sa)=D(a)=e$, while $D(sD(a))=D(se)= e$ also.  Similarly, $R(at)=R(et)$ for each $t\in S$, so the right congruence condition holds.  Also, $D(ef)=D(f)=f=R(f)=R(ef)$, and $D(fe)=D(a)=e=R(a)=R(fe)$.  Finally, if $s,t\in D(S)$, $sD(t)R(s)t=stst=st$ since $S$ is a band, while if $s=a$, $sD(t)R(s)t=aD(t)et=aet=at=st$, and if $t=a$, then $sD(t)R(s)t=seR(s)a=sea=sa=st$.  So by Corollary \ref{strongmatch}, $S$ satisfies the strong match-up conditions.  But $fe=a$ so $D(S)$ is not a band, and so $S$ is not localisable.  
\medskip

As noted above, it is not easy to find natural examples of localisable semigroups that are not also Ehresmann semigroups, and the same therefore applies to cat-semigroups satisfying the strong match-up conditions.  However, our methods require only the strong match-up conditions and we therefore make use of these.

\begin{pro}  \label{match}
Suppose a cat-semigroup $S$ satisfies the first law in the strong match-up conditions -- $R(sD(t))=D(R(s)t)$. Then 
$$st=s|D(t)\circ R(s|D(t))|t=s|D(R(s)|t)\circ R(s)|t\mbox{ in }{\mc C}(S).$$  
Moreover, the corresponding arguments in each of these category products coincide if and only if $S$ is localisable.
\end{pro}
\begin{proof}
Suppose $S$ satisfies the law $R(sD(t))=D(R(s)t)$.  We saw in Proposition \ref{match0} that the match-up conditions are satisfied, and so $st$ may be expressed in ${\mc C}(S)$ in the two ways indicated.  If $S$ is localisable, then it satisfies the congruence conditions and $D(S)$ is a band, so $sD(R(s)t)=sD(R(s)D(t))=sR(s)D(t)=sD(t)$, and similarly, $R(sD(t))t=R(s)t$, so the arguments in the two category products coincide.  Conversely, if $sD(R(s)t)=sD(t)$ and $R(sD(t))t=R(s)t$ for all $s,t$, then because $S$ satisfies the congruence conditions and $D(ef)=R(ef)$ for all $e,f\in D(S)$ by Lemma \ref{stronglem}, we have that for all $e,f\in D(S)$, 
$$D(ef)=D(eef)D(ef)=D(eD(ef))D(ef)=R(eD(ef))D(ef)$$
$$=R(eD(D(ef)))D(ef)=R(e)D(ef)=eD(ef)=eD(R(e)f)=eD(f)=ef,$$
so $D(S)$ is a band, and so $S$ is localisable.
\end{proof}

\subsection{Left semi-localisable semigroups} \label{subsec:leftsemiloc}

One-sided semi-localisability is natural because of the following.

\begin{pro}  \label{bandleftmatch}
The class of cat-semigroups satisfying the left match-up condition and in which the projections form a band is the variety of left semi-localisable precat-semigroups.
\end{pro}
\begin{proof}
Let $S$ be a cat-semigroup in which $D(S)$ is a band.  If $S$ is left semi-localisable, then for all $s\in S$ and $e\in D(S)$, $R(se)=R(R(se)e)=R(se)e$ since $D(S)$ is a band.  Hence, for $s,t\in S$, 
$$D(R(sD(t))t)=D(R(sD(t))D(t))=D(R(sD(t)))=R(sD(t)),$$
so the left match-up condition holds.  The converse follows from Proposition \ref{MUC2}.
\end{proof}

From the above and Proposition \ref{match0}, we obtain the following.

\begin{cor}  \label{loc3way}
For a cat-semigroup $S$, the following are equivalent:
\bi
\item $S$ is localisable;
\item $S$ satisfies the strong match-up conditions and $D(S)$ is a band;
\item $S$ satisfies the match-up conditions and $D(S)$ is a band.
\ei
\end{cor}

There are natural examples of non-localisable, left semi-localisable semigroups, but these have additional properties and are considered next.  

\subsection{The D-ample property and left restriction semigroups with range}  \label{subsec:Dample}

The D-ample law has been considered by many authors in a range of settings, but originated in the work of Fountain \cite{foun}, where it was stated in a way equivalent to that given here if the left congruence condition is assumed.

\begin{pro}  \label{Dampleband}
If $S$ is a D-ample precat-semigroup, then $D(S)$ is a band. 
\end{pro}
\begin{proof}
For all $e,f\in D(S)$, $ef=eD(f)=D(ef)e=D(ef)D(e)=D(D(ef)e)D(ef)=D(eD(f))D(ef)=D(ef)D(ef)=D(ef)$. 
\end{proof}

Note that the above proof only requires the precat-semigroup properties of $D$, not $R$.

In \cite{constell}, a {\em left restriction semigroup} is defined to be a unary semigroup with unary operation $D$ satisfying the following laws:
\bi
\item $D(x)x=x$;
\item $D(x)D(y)=D(y)D(x)$;
\item $D(D(x)y)=D(x)D(y)$;
\item $xD(y)=D(xy)x$.
\ei
Equivalent formulations have also been used. Other laws follow easily, such as the laws $D(xy)D(x)=D(xy)$ and $D(xy)=D(xD(y))$, and the fact that $D(S)$ is a band (from the proof of Proposition \ref{Dampleband}) and hence a semilattice using the second law above. 

Recall that the precat-semigroup $(S,\cdot,D,R)$ is a left restriction semigroup with range if $(S,\cdot,D)$ is a left restriction semigroup and $R$ satisfies $R(xy)R(y)=R(xy)$.
The semigroup $PT(X)$ of partial functions on $X$ is a left restriction semigroup with range, but satisfies the right congruence condition and is therefore localisable.  A non-localisable example follows shortly.

Let $S$ be a left restriction semigroup with range.  Under the usual semilattice order on $D(S)$ given by $e\leq f$ when $e=ef$, it can be shown that $R(s)$ is the smallest $e\in D(S)$ such that $se=s$ (since if $se=s$ then $R(s)e=R(se)e=R(se)R(e)=R(se)$); we use this in the proof to follow.

\begin{thm}  \label{leftrestll}
A left restriction semigroup with range is nothing but a left semi-localisable semigroup which is D-ample and in which the band $D(S)$ is a semilattice.
\end{thm}
\begin{proof}
Suppose $S$ is a left restriction semigroup with range.  We must show $S$ is left semi-localisable.  But $D(S)$ is a band and the left congruence condition is satisfied, so it remains to show that the right weak congruence condition holds in $S$.

First note that for all $s,t\in S$, $st R(R(s)t)=s(R(s)t)R(R(s)t)=sR(s)t=st$, so $R(st)\leq R(R(s)t)$ (since for all $x\in S$, $R(x)$ is the smallest $e\in D(S)$ for which $xe=x$). 

Now for $s,t\in S$, we have that $sD(t)=D(st)s=D(stR(st))s=sD(tR(st))$, and so $sD(t)=sD(t)D(tR(st))$ (since $D(tR(st))\leq D(t)$), from which it follows that $R(sD(t))\leq D(tR(st))$.  Hence,
$$R(sD(t))tR(st)=R(sD(t))D(tR(st))t=R(sD(t))t,$$
so again we must have that $R(R(sD(t))t)\leq R(st)=R((sD(t))t)\leq R(R(sD(t))t)$ from what was shown initially, so all are equal and so $R(st)=R(R(sD(t))t)$.

Conversely, suppose  $S$ is left semi-localisable, D-ample, and the band $D(S)$ is a semilattice.  Then it is a left restriction semigroup with respect to multiplication and $D$, the third law following from the left congruence condition and the fact that $D(S)$ is a band.  Finally, for all $x,y\in S$, we have that
\bea
R(xy)&=&R((xy)R(y))\\
&=&R(R(xyD(R(y)))R(y))\mbox{ by the right weak congruence condition}\\
&=&R(R(xy)R(y))R(y))\\
&=&R(R(xy)R(y))\\
&=&R(xy)R(y)\mbox{ since $D(S)$ is a band,}
\eea
as required.
\end{proof}

A natural example of a left restriction semigroup with range that is not localisable is the semigroup $Rel^d(X)$ of binary relations on the set $X$, equipped with domain, range and {\em demonic composition} $\circledast$, defined as follows: for all $\rho,\tau\in Rel^d(X)$, 
$$\rho\circledast\tau=\{(x,y)\in \rho\tau\mid \mbox{ for all } z\in X, (x,z)\in \rho\Rightarrow z\in \dom(\tau)\},$$
where $\rho\tau$ denotes the usual ``angelic" composition of $\rho,\tau\in Rel(X)$.  As noted by various authors, $\circledast$ is associative; indeed $(Rel(X),\circledast,D)$ is a left restriction semigroup (for example, see \cite{DAD}), and hence satisfies the left congruence condition, but it does not satisfy the right congruence condition in general as easy examples show.  Roughly speaking, angelic composition is important when modelling partial correctness of ``non-deterministic" programs, and demonic for total correctness (which acounts for program termination); see \cite{DAD} and \cite{hms} for more details.

\section{Categories with biaction and the match-up conditions}

When a cat-semigroup is equipped with the left and right actions given by $e|s=es$ and $s|e=se$ for all $e\in D(S)$ and $s\in S$, as defined earlier, not all of the transcription category laws will be satisfied, but some always are.

\begin{dfn}
A category equipped with a left and right action of $C^0$ on $C$, denoted by $e|s$ and $s|e$ for all $s\in C$ and $e\in C^0$, is said to be a {\em category with biaction} if it satisfies (TC1), (TC2) and (TC6) in the definition of a transcription category.  
\end{dfn}

Transcription categories arise as the categories determined by localisable semigroups.  More generally, we have the following easily checked observation.

\begin{prodef}
The category ${\mc C}(S)$ determined by the cat-semigroup $S$ is a category with biaction if we define $e|s=es$ and $s|e=se$ for all $s\in C$ and $e\in C^0$; we call this the {\em category with biaction determined by $S$}.
\end{prodef}

The law (TC6) allows us to write ``$e|s|f$" without ambiguity, where $e,f\in C^0$ and $s\in C$, and we sometimes do this in what follows.  

If $S$ is a cat-semigroup for which the left match-up condition holds then the semigroup product in the cat-semigroup $S$ may be retrieved from the category with biaction structure of ${\mc C}(S)$ by virtue of the fact that $st=(s|D(t))\circ (R(s|D(t))|t)$.  But in general, it turns out that even if both congruence conditions hold in the cat-semigroup $S$, this is not enough to ensure that the semigroup structure of $S$ is captured by ${\mc C}(S)$.

\begin{eg}
Two distinct cat-semigroups determining the same category with biaction. 
\end{eg}
\vspace{-5pt}
Let $X=\{1,2,3,4\}$, and in $PT(X)$, set 
$$e=\{(1,1),(2,2),(3,3)\}, f=\{(2,2),(3,3),(4,4)\}, b=\{(2,2),(3,3)\}.$$
First, let $S=\{e,f,b,0\}$ where $0=\emptyset$, the empty function.  Then $S$ is a subsemigroup of $PT(X)$ with multiplication table as follows:

$$	\begin{array}{c|cccc}
	\cdot&0&e&f&b\\
	\hline
	0&0&0&0&0\\
	e&0&e&b&b\\
	f&0&b&f&b\\
	b&0&b&b&b
	\end{array}.
$$

Now instead let $T=\{e,f,b,0\}$ where this time $0=\{(2,3),(3,2)\}$ with $e,f,b$ unchanged.  This time, the multiplication table is easily seen to be as follows:

$$	\begin{array}{c|cccc}
	\cdot&0&e&f&b\\
	\hline
	0&b&0&0&0\\
	e&0&e&b&b\\
	f&0&b&f&b\\
	b&0&b&b&b
	\end{array}.
$$
The only difference is that $0^2=0$ in $S$, whereas $0^2=b$ in $T$.  

We can define precat-semigroup structure on both $S$ and $T$ by setting $D(S)=D(T)=\{e,f\}$, with $D(0)=D(b)=e$ and $R(0)=R(b)=f$ (and $D(e)=R(e)=e, D(f)=R(f)=f$).  This is easily verified to satisfy the laws.  Considering the law $D(xy)=D(xD(y))$ in either $S$ or $T$, if either side evaluates to $f$, this forces $x=y=f$ and so both sides equal $f$.  A dual argument applies to the law $R(xy)=R(R(x)y)$.  Hence both $S$ and $T$ satisfy the congruence conditions (and so are cat-semigroups).

Note that all products in ${\mc C}(S)$ and ${\mc C}(T)$ agree since $0\circ 0$ does not exist in either, with every other product existing in one exactly when it exists in the other (and all are equal since they are equal in the two semigroups).  Moreover the two biactions agree since they arise from semigroup products between elements at least one of which is not $0$.  So we have that ${\mc C}(S)={\mc C}(T)$ as categories with biaction, yet $S$ and $T$ are not even isomorphic as semigroups.

As we shall see, there are cat-semigroups of interest in which (TC4) fails in the category with biaction they determine, but in which one of (TC4a) and (TC4b) holds.  

\begin{pro}  \label{TC4}
The category with biaction ${\mc C}(S)$ determined by the cat-semigroup $S$ satisfies (TC4) if and only if $S$ satisfies the match-up conditions.
\end{pro}
\begin{proof}
Suppose $S$ satisfies the match-up conditions.  In particular, it satisfies the left match-up condition, and so if $R(s)=D(t)$ and $e\in D(S)$, then 
$$D(R(es)t)=D(R(esR(s))t)=D(R(esD(t))t)=R(esD(t))=R(esR(s))=R(es).$$ 
So $(e|s)\circ (R(e|s)|t)$ exists in ${\mc C}(S)$, and obviously equals $e|(s\circ t)$, so ${\mc C}(S)$ satisfies (TC4a).  By dualising, we infer that the right match-up condition implies (TC4b), and hence that the match-up conditions together imply (TC4).

Conversely, suppose ${\mc C}(S)$ satisfies (TC4); hence, for all $s,t\in S$ for which $R(s)=D(t)$ and for all $e\in D(S)$, $R(es)=D(R(es)t)$ and dually, $D(te)=R(sD(te))$.  But $R(D(t))=D(t)$ for any $t\in S$, so for all $y\in S$ and $e\in D(S)$,
\ben[ (1)]
\item $D(R(eD(y))y)=R(eD(y))$, and dually,
\item $R(yD(R(y)e))=D(R(y)e)$.
\een  
By the first cat-semigroup law applied to (1) above, 
\ben[ (3)]
\item $D(eD(y))=D(eD(y)R(eD(y))y)=D(eD(y)y)=D(ey)$.
\een  
In (2) above, let $e=D(x)$ to give $R(yD(R(y)D(x)))=D(R(y)D(x))$, and so from (3), $R(yD(R(y)x))=D(R(y)x)$, which is the right match-up condition.  We dualise to obtain the result.
\end{proof}

It follows from the above proof that the left match-up condition holding on the cat-semigroup $S$ implies (TC4a) holds on ${\mc C}(S)$, but the converse fails.

\begin{pro} \label{TC4a}
The cat-semigroup $S$ satisfies the left match-up condition if and only if it satisfies $R(st)=D(R(st)R(t))$ and ${\mc C}(S)$ satisfies (TC4a).  
\end{pro}
\begin{proof}
If $S$ satisfies the left match-up condition, then the argument given in the first part of the proof of Proposition \ref{TC4} establishes that ${\mc C}(S)$ satisfies (TC4a).    

Conversely, suppose $S$ satisfies the law $R(st)=D(R(st)R(t))$ and ${\mc C}(S)$ satisfies (TC4a). Now for $x,y\in S$, because $D(y)=R(D(y))$, we get $R(eD(y))=D(R(eD(y))y)$ for any $e\in D(S)$ upon using (TC4a), so by the first cat-semigroup law, $$D(eD(y))=D(eD(y)R(eD(y))y)=D(ey).$$  So letting $e=R(xD(y))$, we get $$D(R(xD(y))y)=D(R(xD(y))D(y))=D(R(xD(y))R(D(y)))=D(R(xD(y)))=R(xD(y)).$$  
Hence the left match-up condition is satisfied.
\end{proof}

Let $S=\{0,e,1\}$ be the three-element semilattice with zero $0$ and identity $1$, and define $D(0)=e$ with $R(0)=1$, and $D(x)=x$ for $x\neq 0$.  It is tedious but routine to check that $S$ is a cat-semigroup with (TC4a) holding on ${\mc C}(S)$, and $R(0e)=R(0)=1$, yet $D(R(0e)R(e))=D(1e)=D(e)=e$, so the additional condition cannot be dispensed with in Proposition \ref{TC4a}.

If $S$ is a cat-semigroup satisfying the strong match-up conditions, then (TC3) does not in general make sense in ${\mc C}(S)$ since $D(S)$ may not be a band, and (TC5) fails in general.  However, (TC4) does hold, by Propositions \ref{TC4a} and \ref{match0}.  

For any class of precat-semigroups, the natural morphism notion is of course semigroup homomorphism that respects $D$ and $R$.  For categories with biaction, the natural notion is as follows (generalising \cite{fitzkin}).

\begin{dfn}  \label{funcdef}
A {\em biaction functor} is a functor $\psi:C_1\rightarrow C_2$ between categories with biaction that satisfies, for all $e\in D(C_1)$ and $s\in C_1$, $\psi(e|s)=\psi(e)|\psi(s)$ and $\psi(s|e)=\psi(s)|\psi(e)$.
\end{dfn}

\section{The ESN theorems}  \label{sec:general}

In this section, our most general ESN theorems are obtained.  We identify the categories with biaction that arise from cat-semigroups equipped with (in order of decreasing generality)
\ben[(i)] 
\item the left match-up condition only,
\item both match-up conditions, and
\item the strong match-up conditions.  
\een
Each of these classes of cat-semigroups is a variety of precat-semigroups, and for each we obtain an isomorphism between the class of precat-semigroups (viewed as a category) and a particular class of categories with biaction (equipped with biaction functors and hence also viewed as a category).  

The description of the categories with biaction arising from Case (i) given in Section \ref{sec:general} includes the condition that the relevant pseudoproduct be associative.   Although this is a first-order condition expressible solely in the language of categories with biaction, it is somewhat unsatisfactory.  Consequently, the special case of left semi-localisable semigroups is considered in Subsection \ref{subsec:special}.  The corresponding categories with biaction are simply transcription categories with some laws missing or strictly weakened.  

For cat-semigroups satisfying both match-up conditions as in Case (ii), we shall show that the corresponding class of categories with biaction has a description in terms of (TC4) and a variant of it -- associativity of the pseudoproduct follows even in the general case.  Case (iii) then builds on this case.

If $S$ is any of the above three types of cat-semigroups, at least the left match-up condition holds, and so it is clear that the category with biaction ${\mc C}(S)$ determined by $S$ satisfies the following property:
\ben [(LMU)]
\item $R(s|D(t))=D(R(s|D(t))|t)$ for all $s,t$.
\een

There is an obvious dual condition:
\ben [(RMU)]
\item $D(R(s)|t)=R(s|D(R(s)|t))$ for all $s,t$.
\een

\begin{dfn}
If a category with biaction $C$ satisfies (LMU) then we define the {\em left pseudoproduct} $\otimes_l$ on $C$ by setting, for all $s,t\in C$, 
$$s\otimes_l t=s|D(t)\circ R(s|D(t))|t,$$ 
and we let ${\mc S}(C)$ be the algebra $(S,\otimes_l,D,R)$, the {\em left extension of $S$}. The {\em right pseudoproduct} $\otimes_r$ is defined dually, if $C$ satisfies (RMU).
\end{dfn}

Clearly, if $S$ is a cat-semigroup satisfying the left match-up condition, then $\otimes_l$ in ${\mc C}(S)$ agrees with the semigroup product on $S$, and dually for the right match-up condition and $\otimes_r$.

\subsection{Categories and the left/right match-up condition}  \label{subsec:leftext}

\begin{dfn}  \label{leftextensible}
Let $C$ be a category with biaction.  We say it satisfies condition (TC4L) if it satisfies (TC4a), and
\ben [(TC4b$^\prime$)]
\item for all $a,b\in C$, if $a\circ b$ exists then for all $e\in C^0$, 
so does $(a|D(b|e))\circ R(a|D(b|e))|(b|e)$, and it equals $(a\circ b)|e$.  
\een
We define condition (TC4R) dually in terms of (TC4a$^\prime$) (defined dually to (TC4b$^\prime$) in the obvious way) and (TC4b).
\end{dfn}

\begin{pro}  \label{catle}
If $S$ is a cat-semigroup satisfying the left match-up condition, then ${\mc C}(S)$ satisfies (LMU) and (TC4L), and $\otimes_l$ coincides with the product on $S$.
\end{pro}
\begin{proof}
From Proposition \ref{TC4a}, $C={\mc C}(S)$ is a category with biaction satisfying (TC4a), and (LMU) follows from the left match-up condition for $S$ upon applying the biaction definition.  For (TC4b$^\prime$), first note that if $a\circ b$ exists then $R(a)=D(b)$ and so for all $e\in C^0$, $R(a|D(b|e))=D(R(a|D(b|e))|(b|e))$ by (LMU), and so the category product $(a|D(b|e))\circ R(a|D(b|e))|(b|e)$ exists and must equal $aD(be)R(aD(be))be=aD(be)be=abe=(a\circ b)|e$.  That $s\otimes_l t=st$ is immediate.
\end{proof}

Next is a useful result for what follows.

\begin{lem}  \label{leftcongcat}
If $C$ is a category with biaction and (TC4a) holds, then for all $s\in C$ and $e\in C^0$, $D(e|s)=D(e|D(s))$.
\end{lem}
\begin{proof}
From (TC4a), $e|s=e|(D(s)\circ s)=e|D(s)\circ R(e|D(s))|s$, so $D(e|s)=D(e|D(s))$.
\end{proof}

Obviously, if a category with biaction satisfies (TC4L) and (TC4R), then it satisfies (TC4), but the converse holds also.

\begin{pro}  \label{extensible}
A category with biaction satisfies (TC4) if and only if it satisfies (TC4L) and (TC4R), and in this case it satisfies both (LMU) and (RMU) as well.
\end{pro}
\begin{proof}
Let $C$ be a category with biaction satisfying (TC4).  From (TC4a), for $s,t\in C$, we have that $R(s)|(D(t)\circ t)=R(s)|D(t)\circ R(R(s)|D(t))|t$ exists and so $R(R(s)|D(t))=D(R(R(s)|D(t))|t)$, so by the dual of Lemma \ref{leftcongcat}, $R(s|D(t))=D(R(s|D(t))|t)$, and so  (LMU) holds.  Dually, (RMU) holds also.

It remains to show that (TC4b$^\prime$) holds.  But if $R(a)=D(b)$ and $e\in C^0$, using (RMU) we have
$$R(a|D(R(a)|(b|e)))=D(R(a)|(b|e))=D(D(b)|(b|e))=D((D(b)|b)|e)=D(b|e).$$  Hence, 
$R(a|D(R(a)|(b|e))|(b|e)=D(b|e)|(b|e)=b|e$.  So from (TC4b), 
$$(a\circ b)|e=a|D(b|e)\circ b|e=a|D(b|e)\circ R(a|D(R(a)|(b|e))|(b|e).$$
Dually, (TC4a$^\prime$) also holds.

The converse has already been noted.
\end{proof}

It follows that each of (TC4L) and (TC4R) generalises (TC4).  Shortly, we prove a kind of converse to Proposition \ref{catle}, but first we need some preliminary results pertaining to categories satisfying the conditions of that result.

\begin{lem}  \label{TC7}
If $C$ is a category with biaction satisfying (LMU) and (TC4L), then for all $x\in C$ and $e\in C^0$, $x|e=(x|e)|e$.
\end{lem}
\begin{proof}
For $x\in C$ and $e\in C^0$, and letting $A=R(x|D(R(x)|e))|R(x)$, we shall prove the following in turn:
\ben[ (1)]
\item $x|e=x|D(R(x)|e)\circ A|e$;
\item $e|e=e$;
\item $A|e=(A|e)|e$;
\item $R(x|e)=R(A|e)$;
\item $(x|e)|e=(x|e)\circ R(x|e)|e$.
\een

First, using (TC4b$^\prime$) we have that
$$x|e=(x\circ R(x))|e=x|D(R(x)|e)\circ R(x|D(R(x)|e))|(R(x)|e),$$
and then applying (TC6) to the second term above gives (1).  For (2), $e|e=e|R(e)=e$ by (TC2).  For (3), letting $s=x|D(R(x)|e)$, we have that $A=R(s)|R(x)$, and so, using (TC6) and (2) gives
\bea 
(A|e)|e&=&((R(s)|R(x))|e)|e\\
&=&(R(s)|(R(x)|e))|e\\
&=&R(s)|((R(x)|e)|e)\\
&=&R(s)|(R(x)|(e|e))\\
&=&R(s)|(R(x)|e)\\
&=&(R(s)|R(x))|e\\
&=&A|e.
\eea
Applying (Cat4) to (1) gives (4).   Finally, for $x\in C$ and $e\in C^0$,
\bea
(x|e)|e&=&(x|e)|D(R(x|e)|e)\circ R((x|e)|D(R(x|e)|e))|(R(x|e)|e)\mbox{ by (TC4b$^\prime$)}\\
&=&(x|e)|R(x|e)\circ R((x|e)|R(x|e))|(R(x|e)|e)\mbox{ by (LMU)}\\
&=&(x|e)\circ R(x|e)|(R(x|e)|e)\mbox{ by (TC2)}\\
&=&(x|e)\circ (R(x|e)|R(x|e))|e\mbox{ by (TC6)}\\
&=&(x|e)\circ R(x|e)|e\mbox{ by (2)},
\eea
establishing (5).

Hence, for $x\in C$ and $e\in C^0$, 
\bea
(x|e)|e&=&(x|e)\circ R(x|e)|e\mbox{ by (5)}\\
&=&x|D(R(x)|e)\circ A|e\circ R(A|e)|e\mbox{ by (1) and (4)}\\
&=&x|D(R(x)|e)\circ (A|e)|e\mbox{ by (5) with $x$ replaced by $A$}\\
&=&x|D(R(x)|e)\circ A|e\mbox{ by (3)}\\
&=&x|e\mbox{ by (1),}\\
\eea
as required.
\end{proof} 

\begin{pro}   \label{pp}  
Let $C$ be a category with biaction satisfying (LMU) and (TC4L).  Then ${\mc S}(C)$ satisfies the following laws: for all $x,y\in C$ and $e\in C^0$, $e\otimes_l x=e|x$, $x\otimes_l e=x|e$, $x\otimes_l y=x\circ y$ whenever $R(x)=D(y)$, laws (CS1) to (CS6) for cat-semigroups, and the left match-up law $R(s\otimes_l D(t))=D(R(s\otimes_l D(t))\otimes_l t)$.  
\end{pro}
\begin{proof} 
For $x\in C$ and $e\in C^0$, using (TC4a) we obtain
$$e|x=e|(D(x)\circ x)=e|D(x)\circ R(e|D(x))|x=e\otimes_l x.$$

Again, for $x\in C$ and $e\in C^0$, 
\bea
x|e &=& (x|e)|e\mbox{ by Lemma \ref{TC7}}\\ 
&=& (x|e\circ R(x|e))|e\\
&=& A\circ R(A)|(R(x|e)|e)\mbox{ by (TC4b$^\prime$),}
\eea
where $A=(x|e)|D(R(x|e)|e)$, which equals $(x|e)|R(x|e)=x|e$ by (LMU), and so
\bea
x|e &=& x|e\circ R(x|e)|(R(x|e)|e)\\
&=& x|e\circ (R(x|e)|R(x|e))|e\mbox{ by (TC6)}\\
&=& x|e\circ R(x|e)|e\mbox{, by (2) in the proof of Lemma \ref{TC7}}\\
&=&(x|D(e))\circ R(x|D(e))|e\\
&=& x\otimes_l e.
\eea
If $R(x)=D(y)$, then 
$$x\otimes_l y=x|D(y)\circ R(x|D(y))|y=x|R(x)\circ R(x|R(x))|y=x\circ R(x)|y=x\circ D(y)|y=x\circ y.$$

We turn to the precat-semigroup laws (CS1)--(CS5). By (2) in the proof of Lemma \ref{TC7}, for $e\in C^0$, we have $$e\otimes_l e=e|D(e)\circ R(e|D(e))|e=e|R(e)\circ R(e|e)|e=e\circ R(e)|e=R(e)|e=e|e=e,$$ establishing (CS1).  Now (CS2) and (CS3) are immediate, and for (CS4) and (CS5), observe that for all $x\in C$, $D(x)\otimes_l x=D(x)|x=x$, and similarly $x\otimes_l R(x)=x$.  

For the cat-semigroup laws in (CS6), if $x,y\in C$ are such that $R(x)=D(y)$ then $$D(x\otimes_l y)=D(x|D(y))=D(x|R(x))=D(x),$$ and $$R(x\otimes_l y)=R(R(x|D(y))|y)=R(R(x|R(x))|y)=R(R(x)|y)=R(D(y)|y)=R(y).$$ 

From (LMU) and what we have already shown, we have that
$$R(xD(y))=R(x|D(y))=D(R(x|D(y))|y)=D(R(xD(y))y),$$ 
proving the left match-up condition.
\end{proof}

\begin{cor}  \label{ppcor}
If $C$ is a category with biaction satisfying (LMU) and (TC4L), and $\otimes_l$ is assocative, then ${\mc S}(C)$ is a cat-semigroup satisfying the left match-up condition.
\end{cor}

Condition (LMU) may be viewed as a replacement for (TC5) in the axioms of transcription categories, with (TC4L) replacing (TC4).  A nine-element counterexample was found using {\em Mace4} \cite{P9M4}, showing that (LMU) is not redundant in the presence of the other laws.  

The associativity assumption in Corollary \ref{ppcor} is necessary.

\begin{eg} A category with biaction satisfying (LMU) and (TC4L) in which $\otimes_l$ is not associative.  \label{notfully}
\end{eg}
\vspace{-5pt}

Let $C=\{s,e,f\}$ be the category in which $C^0=\{e,f\}$ and $D(s)=R(s)=f$, with $s\circ s=f$.  (This fully specifies $C$ as a category.)  Define a biaction on $C$ as follows: $s|e=s|f=s$, $e|s=e|f=e$, $f|s=s$, $f|e=f$, $e|e=e$ and $f|f=f$.  It is a tedious but routine exercise to verify that $C$ is category with biaction and satisfies (LMU) and (TC4L).  But $(s\otimes_l e)\otimes_l s=(s|e)\otimes_l s=s\otimes_l s=s\circ s=f$, whereas $s\otimes_l(e\otimes_l s)=s\otimes_l e=s|e=s$, so $\otimes_l$ is not associative.
\medskip

Recall the definition of a biaction functor as in Definition \ref{funcdef}.

\begin{thm}  \label{big}
The category of cat-semigroups satisfying the left match-up condition is isomorphic to the category of categories with biaction satisfying (LMU), (TC4L) and associativity of $\otimes_l$, with morphisms being biaction functors.
\end{thm}
\begin{proof}
Let $S$ be a cat-semigroup satisfying the left match-up condition.  Then by Proposition \ref{catle}, ${\mc C}(S)$ satisfies the stated conditions, and ${\mc S}{\mc C}(S)=S$.  

Conversely, if $C$ is a category with biaction satisfying the stated conditions, then by Proposition \ref{pp}, ${\mc S}(C)$ is a cat-semigroup satisfying the left match-up condition, in which $s\circ t=st$ whenever $R(s)=D(t)$, $e|s=e\otimes_l s$ and $s|e=s\otimes_l e$ for all $s\in C$ and $e\in C^0$.  These agree with the definitions of the category product and biaction operations in ${\mc C}{\mc S}(C)$, which therefore equals $C$.

We move to the functorial properties. Suppose $f:S_1\rightarrow S_2$ is a semigroup homomorphism respecting $D,R$.  It is immediate that $f:{\mc C}(S_1)\rightarrow {\mc C}(S_2)$ (defined as for $f$ on the underlying sets, so we use the same name for it) is a biaction functor, since the category product and the biactions are simply special cases of semigroup product.  Conversely, if $F:C_1\rightarrow C_2$ is a category with biaction functor and $C_1,C_2$ satisfy (LMU), (TC4L) and associativity of $\otimes_l$, then for $s,t\in {\mc S}(C_1)$, we have $F(D(t))=D(F(t))$ and similarly for $R$, and then
\bea
F(s\otimes_l t)&=&F(s|D(t)\circ R(s|D(t))|t)\\
&=&F(s|D(t))\circ F(R(s|D(t))|t)\\
&=&F(s)|F(D(t))\circ F(R(s|D(t))|F(t)\\
&=&F(s)|D(F(t))\circ R(F(s)|D(F(t)))|F(t)\\
&=&F(s)\otimes_l F(t),
\eea
so $F$ determines a $D$- and $R$-respecting homomorphism ${\mc S}(C_1)\rightarrow {\mc S}(C_2)$.
\end{proof}

The above isomorphism restricts to the one between the category of localisable semigroups and transcription categories implicit in \cite{fitzkin}.  Hence it also restricts to one between the category of Ehresmann semigroups and categories with Ehresmann biaction as in \cite{law2}. 

As far as we know, all previous ESN theorems involving biunary semigroups, both the left and right congruence conditions have been assumed.  Even in settings where biunary semigroups are replaced by semigroups with distinguished idempotents (with multiple candidates for $D(s)$ or $R(s)$ for each $s$), such as that considered in \cite{gw} and \cite{wang1}, versions of the left and right congruence conditions are assumed.  But the above result applies to cat-semigroups that need not satisfy both congruence conditions, such as  $Rel^d(X)$.  In the subsection to follow, we specialise further and consider a class that includes this example.

\subsection{Special case: left semi-localisable semigroups}  \label{subsec:special}

The description of the categories with biaction corresponding to cat-semigroups satisfying the left match-up condition given in Theorem \ref{big} is slightly unsatisfactory, since it simply imposes the requirement of associativity of the left pseudoproduct; on the other hand, this does at least yield a finite first-order axiomatisation expressible in the language of categories with biaction.  Of course, we prefer simpler and more natural laws not specifically referencing the left pseudoproduct yet which force its associativity.  Fortunately, this is possible for the class of left semi-localisable semigroups.  

\begin{thm}  \label{bigsp1}
For left semi-localisable semigroups, the corresponding categories with biaction are those satisfying (TC3), (TC4L) and (TC5a). 
\end{thm}
\begin{proof}
Suppose $S$ is left semi-localisable.  By Proposition \ref{catle}, (TC4L) holds.  That (TC3) and (TC5a) hold in ${\mc C}(S)$ is clear, because $D(S)$ is a band in $S$, the left congruence condition is satisfied, and the biaction in ${\mc C}(S)$ is semigroup multiplication. 

Conversely, suppose $C$ is a category with biaction satisfying (TC3), (TC4L) and (TC5a); in particular then, $D(e|f)=e|f$ for all $e,f\in C^0$.  We first show that (LMU) follows.  

For $x\in C$ and $e\in C^0$, by (TC4b$^\prime$), we have that 
$$x|e=(x\circ R(x))|e=x|(R(x)|e)\circ R(x|(R(x)|e))|(R(x)|e),$$
and so because $R(x|(R(x)|e))|(R(x)|e)\in C^0$, we obtain $x|e=x|(R(x)|e)$ and since $e|e=D(e)|e=e$,
$$R(x|e)=R(x|(R(x)|e))|(R(x)|e)=R(x|e)|R(R(x)|e)=R(x|e)|(R(x)|e)=(R(x|e)|R(x))|e$$
$$=(R(x|e)|R(x))|(e|e)=((R(x|e)|R(x))|e)|e=(R(x|e)|(R(x)|e))|e=R(x|e)|e,$$
so by Lemma \ref{leftcongcat}, for all $x,y\in C$, we have
$$D(R(x|D(y))|y)=D(R(x|D(y))|D(y))=D(R(x|D(y)))=R(x|D(y)),$$
establishing (LMU).

Now suppose $s,t,u\in C$.  Then
$$(s\otimes_l t)\otimes_l u=((s\otimes_l t)|D(u))\circ R((s\otimes_l t)|D(u))|u.$$
But
\bea
(s\otimes_l t)|D(u)&=&(s|D(t)\circ R(s|D(t))|t)|D(u)\\
&=&(s|D(t))|(D(R(s|D(t))|t|D(u))\\
&&\circ R((s|D(t))|(D(R(s|D(t))|t|D(u)))|(R(s|D(t))|t|D(u))\\
&=&A\circ R(A)|(R(s|D(t))|t|D(u))\\
\eea
where 
\bea
A&=&(s|D(t))|(D(R(s|D(t))|t|D(u)))\\
&=&(s|D(t))|(D(R(s|D(t))|D(t|D(u))))\\
&=&(s|D(t))|(R(s|D(t))|D(t|D(u)))\\
&=&(s|D(t))|D(t|D(u))\hspace{2cm}\mbox{$(*)$}\\
&=&s|(D(t)|D(t|D(u)))\\
&=&s|D(D(t)|D(t|D(u)))\\
&=&s|D(D(t)|(t|D(u))\\
&=&s|D((D(t)|t)|D(u))\\
&=&s|D(t|D(u)),
\eea
and so 
\bea
&&R(A)|(R(s|D(t))|t|D(u))\\
&=&R(s|D(t|D(u)))|(R(s|D(t))|t|D(u))\\
&=&R(s|D(t)|D(t|D(u)))|(R(s|D(t))|t|D(u)) \mbox{ by $(*)$ above}\\
&=&R((s|D(t)|R(s|D(t)))|D(t|D(u)))|(R(s|D(t))|(D(t|D(u))|(t|D(u)))\\
&=&R((s|D(t)|R(s|D(t)))|D(t|D(u)))|t|D(u)\\
&&\mbox{ since $e|f\in C^0$ and $R(x|e)|e=R(x|e)$ for all $x\in C$ and $e,f\in C^0$ as above}\\
&=&R((s|D(t))|D(t|D(u)))|t|D(u)\\
&=&R(A)|t|D(u) \mbox{ by $(*)$ above}\\
&=&R(s|D(t|D(u)))|t|D(u).
\eea
Hence $(s\otimes_l t)|D(u)=A\circ R(A)|(R(s|D(t))|t|D(u))=s|D(t|D(u))\circ R(s|D(t|D(u)))|t|D(u)$, and so
\bea
R((s\otimes_l t)|D(u))|u 
&=&R(R(s|D(t|D(u)))|t|D(u))|u\\
&=&R(R(s|D(t|D(u)))|t|D(u)|R(t|D(u)))|u\\
&=&R(R(s|D(t|D(u)))|t|D(u)|R(t|D(u)))|R(t|D(u))|u\\
&=&R(R(s|D(t|D(u))|t|D(u))|R(t|D(u))|u.
\eea
Pulling all this together, 
\bea
(s\otimes_l t)\otimes_l u&=&((s\otimes_l t)|D(u))\circ R((s\otimes_l t)|D(u))|u\\
&=&s|D(t|D(u))\circ R(s|D(t|D(u)))|t|D(u)\\
&&\circ R(R(s|D(t|D(u))|t|D(u))|R(t|D(u))|u.
\eea

On the other hand, 
\bea
s\otimes_l (t\otimes_l u) &=& (s|D(t\otimes_l u))\circ (R(s|D(t\otimes_l u))|(t\otimes_l u))\\
&=&s|D(t|D(u))\circ R(s|D(t\otimes_l u))|((t|D(u))\circ (R(t|D(u))|u))\\
&=&s|D(t|D(u))\circ R(s|D(t|D(u)))|t|D(u)\\
&&\circ R(R(s|D(t|D(u))|t|D(u))|R(t|D(u))|u\\
&=& (s\otimes_l t)\otimes_l u,
\eea
and $\otimes_l$ is associative.   Hence by Corollary \ref{ppcor}, ${\mc S}(C)$ is a cat-semigroup satisfying the left match-up condition.  Then because $e\otimes_l f=e|f$ for all $e,f\in C^0$, it is immediate from the law $D(e|f)=e|f$ for all $e,f\in C^0$ that $D({\mc S}(C))$ is a band.
\end{proof}

Again, the category isomorphism of Theorem \ref{big} restricts to one between the classes mentioned in Theorem \ref{bigsp1}.

\begin{cor}
For left restriction semigroups with range, the corresponding categories with biaction are those as in Theorem \ref{bigsp1} additionally satisfying $e|f=f|e$ for all $e,f\in C^0$, and $s|e=D(s|e)|s$ for all $s\in C$ and $e\in C^0$.
\end{cor}
\begin{proof}
If $S$ is a left restriction semigroup with range, the fact that ${\mc C}(S)$ satisfies the stated conditions is immediate, since the biaction is semigroup multiplication in $S$.  Conversely, if $C$ is a category with biaction in which the two given laws plus those in Theorem \ref{bigsp1} hold, then $e\otimes_l f=e|f=f|e=f\otimes_l e$, so $D({\mc S}(C))$ is a commutative band, and for all $s,t\in C$, $s\otimes_l D(t)=s|D(t)=D(s|D(t))|s=D(s\otimes_l D(t))\otimes_l s=D(s\otimes_l t)\otimes_l s$ by the left congruence condition in ${\mc S}(C)$, so it is D-ample.
\end{proof}

We note that the categories with biaction ${\mc C}(Rel^d(X))$ and ${\mc C}(Rel_X)$ determined by $Rel^d(X)$ and $Rel(X)$ respectively are identical as categories, and the left actions of identities also coincide, but the right actions are different, giving rise to the different cat-semigroup structures.

\subsection{Categories and the match-up conditions}

In light of Proposition \ref{TC4}, one might hope that the categories with biaction corresponding to cat-semigroups satisfying both match-up conditions as in Case (ii) above might be those satisfying (TC4), but this turns out to be too big a class, associativity of the pseudoproduct again being the snag.  However, we find a relatively simple additional condition that yields the hoped-for correspondence.  Case (iii) follows this same path.   

By Proposition \ref{extensible}, a category with biaction satisfying (TC4) also satisfies (TC4L), (TC4R), (LMU) and (RMU), and so both the left and right pseudoproducts may be defined.

\begin{pro}  \label{lrequal}
In a category with biaction satisfying (TC4), for all $s,t\in C$, 
$$s\otimes_l t = s\otimes_r t= s|D(R(s)|D(t))\circ R(s)|D(t)\circ R(R(s)|D(t))|t.$$
\end{pro}
\begin{proof}
Now for $s,t\in C$, upon using the dual of Lemma \ref{leftcongcat} and the defining laws, we have that
\bea
s\otimes_l t&=& s|D(t) \circ R(s|D(t))|t\\
&=& (s\circ R(s))|D(t)\circ R(R(s)|D(t))|t\\
&=& s|(D(R(s)|D(t))\circ R(s)|D(t)\circ R(R(s)|D(t))|t,
\eea
whhich by symmetry must also equal $s\otimes_r t$.
\end{proof}

It therefore makes sense to refer only to ``the pseudoproduct" when (TC4) holds, and henceforth we use the notation ``$\otimes$" for this one operation.  Indeed the above result makes evident that we could from the outset instead have defined the pseudoproduct in a category with biaction satisfying (TC4) to equal the entirely symmetric category {\em triple} product
$$s\otimes t=s|(D(R(s)|D(t))\circ R(s)|D(t)\circ R(R(s)|D(t))|t.$$

Note that Example \ref{notfully} is left/right symmetric and hence satisfies (TC4), showing that (TC4) is not sufficient to imply the associativity of the pseudoproduct.

By dualising  Proposition \ref{catle} and then using Proposition \ref{extensible}, we obtain the following.

\begin{cor}  \label{cate}
If $S$ is a cat-semigroup satisfying the match-up conditions, then ${\mc C}(S)$ satisfies (TC4) and $\otimes$ coincides with the product on $S$.
\end{cor}

From Proposition \ref{extensible} and Corollary \ref{ppcor}, we obtain the following.

\begin{cor} \label{pp1cor}
If $C$ is a category with biaction satisfying (TC4) and $\otimes$ is associative, then ${\mc S}(C)=(C,\otimes,D,R)$ is a cat-semigroup satisfying the match-up conditions.
\end{cor}

By Theorem \ref{big} as well as Corollaries \ref{cate} and \ref{pp1cor}, we obtain the following.

\begin{thm}  \label{big1}
The category of cat-semigroups satisfying the match-up conditions is isomorphic to the category of categories with biaction satisfying (TC4) in which $\otimes$ is associative.
\end{thm}

There is a more elegant way to describe the categories with biaction in the last result, that does not make reference to associativity of $\otimes$, although it is equivalent to two special cases of it in which one of the arguments is in $C^0$.

\begin{pro}  \label{extelegant}
Suppose $C$ is a category with biaction satisfying (TC4).  Then $\otimes$ is associative if and only if $C$ satisfies (TC7), which consists of the following two conditions for all $a,b\in C$ and $e\in C^0$:
\ben [ (TC7a)] 
\item $e|(a|D(b)\circ R(a|D(b))|b)=e|a|D(b)\circ R(e|a|D(b))|b$ and
\item $(a|D(R(a)|b) \circ R(a)|b)|e=a|D(R(a)|b|e)\circ R(a)|b|e$.
\een
\end{pro}
\begin{proof}
First note that the category products in (TC7) both exist, because (LMU) and (RMU) hold by Proposition \ref{extensible}.

Suppose $\otimes$ is associative in $C$.  Then in particular, for all $a,b\in C$ and $e\in C^0$, we have that $e\otimes(a\otimes b)=(e\otimes a)\otimes b$, so by Propositions \ref{pp} and \ref{lrequal}, $e|(a\otimes_l b)=(e|a)\otimes_l b$, which yields (TC7a); we argue dually to give (TC7b).

Conversely, let $C$ be a category with biaction satisfying (TC4) and (TC7).  
Then for $a,b,c\in C$ we have
\bea
(a\otimes b)\otimes c&=& (a\otimes_r b)\otimes_l c\mbox{ by Proposition \ref{lrequal}}\\
&=&(a\otimes_r b)|D(c)\circ R((a\otimes_r b)|D(c))|c\\
&=&(a|D(R(a)|b)\circ R(a)|b)|D(c)\circ R((a\otimes b)|D(c))|c\mbox{ by (TC7b)}\\
&=&a|D(R(a)|b|D(c))\circ R(a)|b|D(c)\circ R(R(a)|b|D(c))|c,
\eea
and by symmetry (using (TC7a), this must also equal $a\otimes (b\otimes c)$.  So $\otimes$ is associative.
\end{proof}

\begin{cor}  \label{corTC7}
The category of cat-semigroups satisfying the match-up conditions is isomorphic to the category of categories with biaction satisfying (TC4) and (TC7).
\end{cor}

In fact, a slightly strengthened form of (TC7) implies (TC4).

\begin{lem}
Let $C$ be a category with biaction.  Then $C$ satisfies both (TC4) and associativity of $\otimes$ if and only if it satisfies (TC7$^\prime$), given by the following conditions: for all $a,b\in C$ and $e\in C^0$,
\ben [ (TC7a$^\prime$)] 
\item $(a|D(b)\circ R(a|D(b))|b$ exists and $e|(a|D(b)\circ R(a|D(b))|b)=e|a|D(b)\circ R(e|a|D(b))|b$;
\item $a|D(R(a)|b) \circ R(a)|b$ exists and $(a|D(R(a)|b) \circ R(a)|b)|e=a|D(R(a)|b|e)\circ R(a)|b|e$.
\een
\end{lem}
\begin{proof}
Evidently (TC7a$^\prime$) and (TC7b$^\prime$) hold in categories with biaction satisfying (TC4) and associativity of $\otimes$, by Proposition \ref{extensible} (ensuring that (LMU) and (RMU) hold) and Proposition \ref{extelegant}.

Conversely, assume (TC7a$^\prime$) and (TC7b$^\prime$) hold in $C$.  Then assuming that $R(a)=D(b)$ in (TC7a$^\prime$) yields (TC4a) as a consequence, and dually for (TC7b$^\prime$) and (TC4b).  Clearly, (TC7a) and (TC7b) follow easily.
\end{proof}

\begin{cor}
The category of cat-semigroups satisfying the match-up conditions is isomorphic to the category of categories with biaction satisfying (TC$7^\prime$).
\end{cor}

Theorem \ref{big1} and its two corollaries can be specialised further as desired.  For example, to give a description of the categories with biaction arising from DRC-semigroups in the sense of \cite{jonesDRC} and \cite{Swang}, one need only add analogs of the reduced property of DRC-semigroups to the category with biaction axioms.  

Indeed, it is of interest to contrast our approach to obtaining an ESN theorem to the one used by Wang in \cite{Swang} for DRC-semigroups.  There, the author took the more traditional order-theoretic approach, in which notions of restriction and co-restriction are defined order-theoretically in a generalised category equipped with a suitably well-behaved partial order consistent with an assumed projection algebra structure on the identities.  In this approach, one only defines the ``restriction" $e|s$ ($e\in C^0, s\in C$) when $e\leq D(s)$ (under the given order), and then $e|s$ is defined to be the unique $t\leq s$ such that $D(t)=e$; dually for the ``co-restriction" $s|e$.  (However, some additional purely algebraic laws involving restriction and co-restriction are needed.)  Hence, fewer products of projections with arbitrary semigroup elements are retained in the partial algebra so determined, compared to the current approach in which the ``biaction" is defined for arbitrary $e,s$.

However, in the approach of \cite{Swang}, when obtaining the generalised category corresponding to a given DRC-semigroup, the partial product $s\cdot t$ is defined to exist (and be $st$) if and only if $R(s)=D(R(s)D(t))$ and $D(t) = R((R(s)D(t))$, for which it is sufficient but certainly not necessary that $R(s)=D(t)$.  Hence, a greater number of general products is retained in this generalised category than in the associated category with biaction.  Moreover, the approach in \cite{Swang} requires the initial specification of both a partial order on the entire generalised category, and a projection algebra structure on its identities; our approach requires neither.  Indeed, in our setting such order structure does not even seem definable in general.

\subsection{Categories and the strong match-up conditions}

We conclude this section by considering the category with biaction analog of the strong match-up condition.

\begin{dfn}  \label{strongextensible}
We define (SMU) to consist of the following two conditions on the category with biaction $C$: 
\ben [(SMU1)]
\item $D(R(s)|t)=R(s|D(t))$ for all $s,t\in C$, and
\item $(e|f)\circ (e|f)=e|f$ for all $e,f\in C^0$.
\een
\end{dfn}

Note that the product in (SMU2) exists, if we assume (SMU1): $R(e|f)=R(e|D(f))=D(R(e)|f)=D(e|f)$ for all $e,f\in C$.  

In this case, the pseudoproduct may be expressed in the form used in \cite{fitzkin}.

\begin{pro}  \label{strongood}
Suppose a category with biaction $C$ satisfies (TC4) and (SMU).  Then 
$$s\otimes t=s|D(t)\circ R(s)|t\mbox{ for all }s,t\in C.$$ 
\end{pro}
\begin{proof}
For all $x,y\in C$, 
$x|D(y)\circ R(x)|y$ exists by (SMU1), and  
\bea
x|D(y)\circ R(x)|y
&=& (x\circ R(x))|D(y)\circ R(x)|(D(y)\circ y)\\
&=& x|D(R(x)|D(y))\circ R(x)|D(y)\circ R(x)|D(y)\circ R(R(x)|D(y))|y\mbox{ by (TC4)}\\
&=& x|D(R(x)|D(y))\circ R(x)|D(y)\circ R(R(x)|D(y))|y\mbox{ by (SMU2)}\\
&=&x\otimes y\mbox{ by Proposition \ref{lrequal},}
\eea
as claimed.
\end{proof}

Observe that Example \ref{notfully} satisfies (SMU), showing that (TC4) and (SMU) are together not sufficient for the associativity of $\otimes$.

\begin{pro}  \label{catse}
If $S$ is a cat-semigroup satisfying the strong match-up conditions, then ${\mc C}(S)$ satisfies (TC4) and (SMU), and $\otimes$ coincides with the product on $S$.
\end{pro}
\begin{proof}
By Proposition \ref{match0}, $S$ satisfies the match-up conditions; hence ${\mc C}(S)$ satisfies (TC4) by Corollary \ref{cate}, and then (SMU) follows easily from the strong match-up conditions for $S$.
\end{proof}

If $C$ is a category with biaction satisfying (TC4) and (SMU), it follows from Propositions \ref{match} and \ref{catse} that the terms $R(x)|y$ and $D(R(x)|y)|y$ occurring in the pseudoproduct and left pseudoproduct are not necessarily equal, even though the two pseudoproducts are equal; these terms are equal when $C$ is a transcription category.

\begin{pro} \label{pp2cor} 
If $C$ is a category with biaction satisfying (TC4), (TC7) and (SMU), then ${\mc S}(C)=(C,\otimes,D,R)$ is a cat-semigroup satisfying the strong match-up conditions.
\end{pro}
\begin{proof}
From Corollary \ref{pp1cor} as well as Propositions \ref{extelegant} and \ref{strongood}, $(C,\otimes,D,R)$ satisfies everything that is claimed except perhaps the strong match-up conditions, although it does satisfy the match-up conditions.  For $s,t\in C$, 
$D(R(s)\otimes t)=D(R(s)|t)=R(s|D(t))=R(s\otimes D(t))$, and
\bea
(s\otimes D(t))\otimes (R(s)\otimes t) &=& (s|D(t))\otimes (R(s)|t)\\
&=&(s|D(t))|D(R(s)|t)\circ R(s|D(t))|(R(s)|t)\\
&=&(s|D(t))|R(s|D(t))\circ D(R(s)|t)|(R(s)|t)\mbox{ by (SMU1)}\\
&=&s|D(t)\circ R(s)|t\\
&=&s\otimes t\mbox{ by Proposition \ref{strongood},}
\eea 
so the strong match-up conditions hold in ${\mc S}(C)$.
\end{proof}

By Corollary \ref{corTC7} and Propositions \ref{strongood}, \ref{catse} and \ref{pp2cor}, we obtain the following.

\begin{thm}  \label{big2}
The category of cat-semigroups satisfying the strong match-up conditions is isomorphic to the category of categories with biaction satisfying (TC4), (TC7) and (SMU).
\end{thm}

From the proof of Proposition \ref{extelegant} we note that (TC7) can be written in slightly simpler form here since $s\otimes t=s|D(t)\circ R(s)|t$ in this case.

\vspace{2cm}

\noindent Tim Stokes\\
Department of Mathematics\\
University of Waikato\\
Hamilton 3216\\
New Zealand.\\
email: tim.stokes@waikato.ac.nz

\end{document}